\documentclass[a4paper]{amsart}
\usepackage[utf8]{inputenc}
\usepackage[english]{babel}
\usepackage[left=2.5cm,right=2.5cm,top=2.5cm,bottom=2.5cm]{geometry}
\usepackage{amsmath,amssymb,latexsym, amsthm, enumitem}
\usepackage{graphicx}
\usepackage{hyperref}

\newcommand{\id}{\,\mathrm{d}}

\newcommand{\inj}{\mathrm{inj}}
\renewcommand{\exp}{\mathop{\rm exp}\nolimits}
\providecommand{\Exp}{\mathop{\rm exp}\nolimits}

\newtheorem{theorem}{Theorem}[section]
\newtheorem{lemma}{Lemma}[section]
\newtheorem{corollary}{Corollary}[section]

\theoremstyle{definition}

\newtheorem{definition}{Definition}[section]
\newtheorem{question}{Question}

\title{D'Atri spaces and the total scalar curvature of hemispheres, tubes and cylinders}
\begin{document}
\author{Bal\'azs Csik\'os}
\address{B. Csik\'os, Dept.\ of Geometry, E\"otv\"os Lor\'and University, Budapest}
\email{\href{mailto:balazs.csikos@ttk.elte.hu}{balazs.csikos@ttk.elte.hu}}

\author{Amr Elnashar}
\address{A. Elnashar, Central European University, Budapest}
\email{\href{mailto:Elnashar\_Amr@phd.ceu.edu}{Elnashar\_Amr@phd.ceu.edu}} 
\author{M\'arton Horv\'ath}
\address{M. Horv\'ath, Dept.\ of Geometry, Budapest University of Technology and Economics, Budapest} 
\email{\href{mailto:horvathm@math.bme.hu}{horvathm@math.bme.hu}}
\date{}
\maketitle
\begin{abstract}
B.~Csik\'os and M.~Horv\'ath proved in \cite{tube2} that if a connected Riemannian manifold of dimension at least $4$ is harmonic, then the total scalar curvatures of tubes of small radius about a regular curve depend only on the length of the curve and the radius of the tube, and conversely, if the latter condition holds for cylinders, i.e., for tubes about \emph{geodesic} segments, then the manifold is harmonic. In the present paper, we show that in contrast to the higher dimensional case, a connected 3-dimensional Riemannian manifold has the above mentioned property of tubes if and only if the manifold is a D'Atri space, furthermore, if the space has bounded sectional curvature, then it is enough to require the total scalar curvature condition just for cylinders to imply that the space is D'Atri. This result gives a negative answer to a question posed by L.~Gheysens and L.~Vanhecke. To prove these statements, we give a characterization of D'Atri spaces in terms of the total scalar curvature of geodesic hemispheres in any dimension. 
\end{abstract}

\section{Introduction}
 
 By H.~Hotelling's theorem \cite{Hotelling}, in the $n$-dimensional Euclidean or spherical space, the volume of a solid tube of small radius about a curve depends only on the length of the curve and the radius of the tube. A.~Gray and L.~Vanhecke \cite{Gray-Vanhecke} extended Hotelling's theorem to rank one symmetric spaces. B.~Csik\'os and M.~Horv\'ath \cite{tube1}, \cite{tube2} showed  that Hotelling's theorem is true also in harmonic manifolds, and conversely,  if a Riemannian manifold has the property that the volume of a solid tube of small radius about a \emph{geodesic segment} depends only on the radius of the tube and the length of the geodesic, then the manifold is harmonic. Using the Steiner-type formula of E.~Abbena, A.~Gray, and L.~Vanhecke \cite{Abbena_Gray_Vanhecke}, the above characterization of harmonic spaces provided further similar characterizations of harmonicity in which the condition on the volume of solid tubes is replaced by analogous conditions either on the surface volume, or on the total mean curvature of the tubular hypersurfaces. If the dimension of the manifold is at least $4$, harmonicity can also be characterized by an analogous property of the total scalar curvature of the tubular hypersurfaces. It was left open in \cite{tube2} whether the restriction on the dimension is necessary in the case of total scalar curvature. L.~Gheysens and L.~Vanhecke \cite[p.~193]{Gheysens_Vanhecke} pointed out that the $3$-dimensional case is different. They also posed the question whether vanishing of the total scalar curvature of tubes about curves in a $3$-dimensional Riemannian manifold implies that the manifold is harmonic. Recall that a $3$-dimensional connected Riemannian manifold is harmonic if and only if it is of constant sectional curvature. 
 
 The goal of the present paper is to fill this gap and characterize $3$-dimensional Riemannian manifolds, in which the total scalar curvature of tubular surfaces of small radii about regular curves, or only about geodesic segments depends only on the length of the central curve and the radius of the tube. One of our main theorems, Theorem \ref{thm:main2} says that a 3-dimensional Riemannian manifold has this property for tubes about arbitrary regular curves if and only if the space is a D'Atri space, furthermore, the total scalar curvature of tubes in a 3-dimensional D'Atri space is constant $0$. 
 
 Recall that a Riemannian manifold is said to be a D'Atri space if the local geodesic reflection with respect to an arbitrary point is volume-preserving. Every harmonic manifold is a D'Atri space, but the family of D'Atri spaces is strictly larger than that of harmonic manifolds even in dimension $3$, as shown by the classification of $3$-dimensional D'Atri spaces by  O.~Kowalski \cite{Kowalski}. In particular, by Theorem \ref{thm:main2}, the answer to the above mentioned question of L.~Gheysens and L.~Vanhecke is negative.
  
 It is a natural question to ask whether the D'Atri property of a $3$-dimensional Riemanian manifold is implied by the weaker assumption that the total scalar curvature of tubular surfaces of small radius about geodesic segments depends only on the length of the geodesic and the radius of the tube. In Theorem  \ref{thm:main3}, we show that the answer is yes, if we assume additionally that the manifold is complete and has bounded sectional curvature, for example if it is compact, or homogeneous. However, the following question remains open.
 
 \begin{question}
 Can we omit the assumptions on completeness and boundedness of the sectional curvature in Theorem  \ref{thm:main3}?
 \end{question}    

The proof of Theorems \ref{thm:main2} and \ref{thm:main3} will be based on Theorem \ref{thm:main1}, which provides some characterizations of D'Atri spaces in terms of the scalar curvature functions of geodesic spheres. It claims, for example, that a Riemannian manifold is a D'Atri space if and only if any two geodesic hemispheres lying on the same geodesic sphere have the same total scalar curvature. There is a strong conjecture proposing that all D'Atri spaces are locally homogeneous. If it is true, then it would complete the classification problem of harmonic manifolds by J.~Heber \cite{Heber}. The conjecture is true in dimension $3$ and is supported by a theorem of P.~G\"unther and F.~Pr\"ufer \cite{Gunther_Prufer} claiming that in a D'Atri space, the volume of small balls depends only on the radius, but not on the center. A positive answer to the following question would be a further support to the conjecture and would sharpen Theorem \ref{thm:main1}.

\begin{question}
	Do small geodesic spheres of the same radius have equal total scalar curvature in a connected D'Atri space?
\end{question}   

\section{Notations}

All manifolds in this paper are assumed to be smooth, connected, and of dimension at least $3$.

Let $(M,\langle\,,\rangle)$ be a Riemannian manifold of dimension $n$. The symbols $\nabla$, $R$, $\rho$, and $\tau$ will denote the Levi-Civita connection, the curvature tensor, the Ricci tensor and the scalar curvature function of $M$, respectively. For a two-dimensional linear subspace $\sigma\subset T_pM$, the sectional curvature in the direction of $\sigma$ will be denoted by $K(\sigma)$.  Let $\mathring{T}M\subseteq TM$ be the domain of the exponential map $\exp\colon \mathring{T}M\to M$ of $M$, $\exp_p \colon \mathring{T}_p M \to M$ be the restriction of $\exp$ to $\mathring{T}_p M=T_pM\cap \mathring{T}M$. 
The injectivity radius at $p$ will be denoted by $\inj(p)$.

For $p\in M$ and $r>0$, we shall denote by $B_p(r)\subset T_pM$ 
and $S_p(r)\subset T_pM$ the closed ball and the 
sphere of radius $r$ centered at the origin $\mathbf 0_p\in T_pM$, respectively. 
The unit sphere $S_p(1)$ will be denoted simply by $S_p$. Denote by $SM=\bigcup_{p\in M}S_p\subset TM$ the total space of the unit sphere bundle of the tangent bundle. 

Associated to a non-zero tangent vector $\mathbf v\in T_pM\setminus\{\mathbf 0_p\}$, we shall consider 
the hemisphere
\[
S^{+}(\mathbf v)=\{\mathbf w\in T_pM \mid \langle \mathbf w,\mathbf v\rangle\geq 0,\, \|\mathbf w\|=\|\mathbf v\|\}. 
\]

When $r< \inj(p)$ and $\|\mathbf v\|<\inj(p)$, we can take the exponential images 
\[
\begin{matrix}
\mathcal S_p(r)=\exp(S_p(r)),&\quad&
&\quad&\mathcal S^+(\mathbf v)=\exp(S^+(\mathbf v)).
\end{matrix}
\]
The set 
$\mathcal S_p(r)$ is the
geodesic sphere of radius $r$ centered at $p$. 
Analogously, the set 
$\mathcal S^+(\mathbf v)$ will be called a
\emph{geodesic hemisphere}.

	For a smooth regular curve $\gamma\colon[a,b]\to M$ and $r>0$, set 
	\[
	T_{\bullet}(\gamma,r)=\{\mathbf v\in TM\mid \exists t\in[a,b]\text{ such that }\mathbf 
	v\in T_{\gamma(t)}M, \mathbf v\perp\gamma'(t),\text{ and }\|\mathbf v\|\leq 
	r\},\]
	and 
	\[
	T_{\circ}(\gamma,r)=\{\mathbf v\in TM\mid \exists t\in[a,b]\text{ such that }\mathbf 
	v\in T_{\gamma(t)}M, \mathbf v\perp\gamma'(t),\text{ and }\|\mathbf v\|=r\}.
	\]
	Assume that $r$ is small enough to guarantee that $T_{\bullet}(\gamma,r)\subset \mathring{T}M $ and the	exponential map is an immersion of $T_{\bullet}(\gamma,r)$ into $M$. Then we define the 
	\emph{solid tube of radius $r$ about $\gamma$} by
	\[\mathcal T_{\bullet}(\gamma,r)=\Exp(T_{\bullet}(\gamma,r)),\]
	while the \emph{tubular hypersurface,} or shortly the \emph{tube of radius $r$ about $\gamma$} is defined as 
	\[\mathcal T_{\circ}(\gamma,r)=\Exp(T_{\circ}(\gamma,r)).\]
	Tubular hypersurfaces about geodesic segments will be called \emph{cylinders}. 

Speaking of geodesic spheres and hemispheres, tubes, and cylinders of small radius $r$, ``small'' will always mean that $r$ satisfies the requirements given above in the definition of these geometric shapes.

The scalar curvature of the geodesic sphere $\mathcal S_p(r)$ at the point $\exp_p(\mathbf v)$ for $\mathbf v\in S_p(r)$ will be denoted by $\tau^S(\mathbf v)$.

	The \emph{total scalar curvature} of a compact submanifold, possibly with boundary, of a Riemannian manifold is the integral of the scalar curvature function of the submanifold over the submanifold with respect to the volume measure induced by the Riemannian metric. The definition can be extended in an obvious way to immersed submanifolds having self-intersections. 

\section{D'Atri spaces and the total scalar curvature of hemispheres}

Recall that the \emph{volume density function} $\theta\colon \mathring{T}M\to \mathbb R$ is defined by the formula 
\[
\theta(\mathbf v)=\|T_{\mathbf v}\exp_p(\mathbf e_1)\wedge\dots\wedge T_{\mathbf v}\exp_p(\mathbf e_n)\|,
\]
where $\mathbf v\in \mathring{T}_pM$,  $(\mathbf e_1,\dots,\mathbf e_n)$ is an orthonormal basis of the Euclidean linear space $T_{\mathbf v}(T_pM)\cong T_pM$, and $T_{\mathbf v}\exp_p$ denotes the derivative map of the exponential map $\exp_p$ at $\mathbf v$.

Given a unit tangent vector $\mathbf u\in SM$, the coefficients $a_k(\mathbf u)$ in the Taylor series $\sum_{k=0}^{\infty}a_k(\mathbf u)r^k$ of the function $\theta(r\mathbf u)$ can be expressed  explicitly in terms of the curvature tensor of $M$. The initial terms are

	\begin{equation}\label{eq:theta_sorfejtes}
	\theta(r\mathbf u)=1-\frac{\rho(\mathbf u,\mathbf u)}{6}r^2-\frac{\nabla_{\mathbf u}\rho(\mathbf u,\mathbf u)}{12}r^3+O(r^4),
\end{equation}
see \cite[Cor.~2.4]{Chen_Vanhecke}.

\begin{definition}
	A Riemannian manifold $M$ is a \emph{D'Atri space} if for any point $p\in M$, the local geodesic symmetry in $p$ is volume-preserving, or equivalently, if for all $p\in M$, there is a ball $B_p(r)\subset \mathring{T}_pM$ such that $\theta(\mathbf v)= \theta(-\mathbf v)$ for all $\mathbf v\in B_p(r)$.  
\end{definition}

The definition implies at once, that in a D'Atri space, all the odd coefficients $a_{2k+1}(\mathbf u)$ in the Taylor series of the function $\theta(r\mathbf u)$ must vanish. The identity $a_3(\mathbf u)\equiv 0$, also called the third Ledger condition $L_3$, means that the Ricci tensor of $M$ is cyclic parallel, i.e., it satisfies the identity
\[
\nabla_X\rho(Y,Z)+\nabla_X\rho(Y,Z)+\nabla_X\rho(Y,Z)=0.
\] 
It was proved by Z.~I.~Szab\'o \cite[Ch.~2, Thm.~1.1]{Szabo2}, that any Riemannian manifold with cyclic parallel Ricci tensor, in particular, every D'Atri space is a real analytic Riemannian manifold, consequently in such manifolds, the function $\theta(r\mathbf u)$ coincides with the sum of its Taylor series $\sum_{k=0}^{\infty}a_k(\mathbf u)r^k$ when $r$ is small. This also gives the equivalence of the D'Atri property to the vanishing of all the odd coefficients $a_{2k+1}$.

The following technical lemma provides a characterization of spaces with cyclic parallel Ricci tensor.

\begin{lemma}\label{lin}
	The following two statements are equivalent for an $n$-dimensional Riemannian manifold:
	\begin{itemize}
		\item[(i)] $\nabla_X\rho(X,X)\equiv 0$ and $\nabla \tau\equiv 0$. 
		\item[(ii)] $\nabla_X\rho(X,X)+c \nabla_X \tau\|X\|^2 \equiv 0$ for some constant $c\neq -2/(n+2)$. 
	\end{itemize}	
\end{lemma}
\begin{proof} It is clear that $(i)\Longrightarrow (ii)$, consider the converse. Polarizing $(ii)$ we get 
	\[\big(\nabla_X\rho(Y,Z)+\nabla_Y\rho(Z,X)+\nabla_Z\rho(X,Y)\big)+c\big(\nabla_X\tau\langle Y, Z\rangle +\nabla_Y\tau\langle Z, X\rangle+\nabla_Z\tau\langle X, Y\rangle\big)\equiv 0.\]	
	To prove $(i)$ at a particular point $p\in M$, choose an orthonormal frame $E_1,\dots,E_n$ around $p$, substitute $Y=Z=E_i$ into the above identity and take sum for $i$. Using the identity $2\textrm{div}\,\rho=\nabla\tau$, this gives
	\begin{align*}0& \equiv \sum_{i=1}^n\Big(\nabla_X\rho(E_i,E_i)+\nabla_{E_i}\rho(X,E_i)+\nabla_{E_i}\rho(E_i,X)\Big)+c\Big(n\nabla_X\tau+2\sum_{i=1}^n\nabla_{\langle X,E_i\rangle E_i}\tau\Big)\\&=\Big(\sum_{i=1}^n\nabla_X\rho(E_i,E_i)\Big)+2\textrm{div}\,\rho(X)+c(n+2)\nabla_X\tau=\Big(\sum_{i=1}^n\nabla_X\rho(E_i,E_i)\Big)+(c(n+2)+1)\nabla_X\tau.
	\end{align*}
Introducing the notation $\omega_i^j(X)=\langle\nabla_XE_i,E_j\rangle$ and using the skew symmetry 
\[
\omega_i^j(X)+\omega_j^i(X)=\nabla_X\langle E_i,E_j\rangle=0,
\]
we also have
	\begin{align*}\nabla_X\tau=\sum_{i=1}^n\nabla_X(\rho(E_i,E_i))=\sum_{i=1}^n\big(\nabla_X\rho(E_i,E_i)\big)-2\sum_{i,j=1}^n\rho(\omega_i^j(X)E_j,E_i)\big)=\sum_{i=1}^n\nabla_X\rho(E_i,E_i).
	\end{align*}
	Hence $(c(n+2)+2)\nabla_X\tau \equiv 0$, which yields $\nabla\tau=0$ and $(i)$.
\end{proof}
 
The following consequence of the Steiner-type formula of  E.~Abbena, A.~Gray, and L.~Vanhecke \cite{Abbena_Gray_Vanhecke} will play a crucial role in the proof of the main theorem of this section. 
\begin{lemma}\label{lem:Steiner_for_spheres}
	The volume density function, and the scalar curvature of geodesic spheres are related by the formula
\begin{equation}\label{eq:Steiner_for_spheres}
	\big(\rho(\gamma'_{\mathbf u}(r),\gamma'_{\mathbf u}(r))+\tau^S(r\mathbf u)-\tau(\gamma_{\mathbf u}(r))\big)\theta(r\mathbf u)=\partial_r^2\theta(r\mathbf u)+2(n-1)\partial_r\theta(r\mathbf u)\frac{1}{r} +(n-1)(n-2)\theta(r\mathbf u)\frac{1}{r^2},
\end{equation}	
where $\mathbf u\in S_pM$ is an arbitrary unit tangent vector, $\gamma_{\mathbf u}$ is the unit speed geodesic with initial velocity $\gamma_{\mathbf u}'(0)=\mathbf u$, $0<r<\inj(p)$.
\end{lemma}

\begin{proof}
Choose an arbitrary open subset $\mathcal U=\exp_p(r U)\subset \mathcal S_p(r)$  of a geodesic sphere, where $U\subset S_p$ is an open subset, and  compute the volume $V_{\mathcal U}(h)$ of the one-sided parallel domain $\bigcup_{r\leq s\leq r+h} \exp_p(sU)$
of height $h$ over $\mathcal U$ in two different ways for $0<r<r+h<\inj(p)$. 
		First, computing the volume by integrating the density function $\theta$ over the corresponding domain in the tangent space, we obtain
	\begin{equation*}
	\begin{aligned}
		V_{\mathcal U}(h)&=\int_U\int_0^{h}\theta((r+t)\mathbf u)(r+t)^{n-1}\id t \id \mathbf u\\
		&=\int_U\int_0^{h}\Big\{\theta(r\mathbf u)r^{n-1}+\partial_r\big(\theta(r\mathbf u)r^{n-1}\big)t+\partial_r^2\big(\theta(r\mathbf u)r^{n-1}\big)\frac{t^2}{2}+O(t^3)\Big\}\id t \id \mathbf u\\
		&=\int_U\Big\{\theta(r\mathbf u)r^{n-1}h+\big\{\partial_r\theta(r\mathbf u)r^{n-1}+(n-1)\theta(r\mathbf u)r^{n-2}\big\}\frac{h^2}{2}\\&\qquad\quad +\big\{\partial_r^2\theta(r\mathbf u)r^{n-1}+2(n-1)\partial_r\theta(r\mathbf u)r^{n-2} +(n-1)(n-2)\theta(r\mathbf u)r^{n-3}\big\}\frac{h^3}{6}+O(h^4)\Big\} \id \mathbf u.
	\end{aligned}
\end{equation*}	
	On the other hand, the Steiner-type formula of E.~Abbena, A.~Gray, and L.~Vanhecke \cite[Thm.~3.5]{Abbena_Gray_Vanhecke} tells us that 
		\begin{equation*}\label{eq:volum_of_parallel_domain2}
			V_{\mathcal U}(h)=\int_U\Big\{h-H(r\mathbf u) \frac{h^2}{2} +	\big(\rho(\gamma'_{\mathbf u}(r),\gamma'_{\mathbf u}(r))+\tau^S(r\mathbf u)-\tau(\gamma_{\mathbf u}(r))\big)\frac{h^3}{6}+O(h^4)\Big\}r^{n-1}\theta(r\mathbf u)\id  \mathbf u,
	\end{equation*}		
		where $H(r\mathbf u)$ is the trace of the Weingerten map of $\mathcal S_p(r)$ at $\gamma_{\mathbf u}(r)$ relative to the normal vector $\gamma_{\mathbf u}'(r)$.

As the two integrals expressing $V_{\mathcal U}(h)$ are equal for any open subset $U\subset S_p$ and any $0<r<\inj(p)$, the integrands must be equal pointwise. Equating the coefficients of $h^3$ in the expansions of the integrands yields the desired identity.
\end{proof}	
	
Now we prove the main theorem of this section.

\begin{theorem}\label{thm:main1} For a Riemannian manifold $(M,\langle\,,\rangle)$, the following statements are equivalent:
	\begin{enumerate}[label=(\roman*)]
	\item \label{item:(i)} $M$ is a D'Atri space.
	\item \label{item:(iii)} The product $\tau^S\theta$ is an even function, i.e., $\tau^S(\mathbf v)\theta(\mathbf v)=\tau^S(-\mathbf v)\theta(-\mathbf v)$ whenever both sides are defined.
	\item \label{item:(iv)} The total scalar curvatures of any two geodesic hemispheres lying on a given geodesic sphere are equal.
\end{enumerate}
\end{theorem}

\begin{proof}
	First we show the implication  
	$\ref{item:(i)}\Longrightarrow \ref{item:(iii)}$. Expressing the function $r\mapsto \tau^S(r\mathbf u)$ for an arbitrary fixed unit tangent vector $\mathbf u\in SM$ with the help of Lemma \ref{lem:Steiner_for_spheres}, we obtain
	\begin{equation}\label{eq:tauS_from_theta}
		\tau^S(r\mathbf u)=\frac{\partial_r^2\theta(r\mathbf u)}{\theta(r\mathbf u)}+2(n-1)\frac{\partial_r\theta(r\mathbf u)}{\theta(r\mathbf u)}\frac{1}{r} +(n-1)(n-2)\frac{1}{r^2} +\tau(\gamma_{\mathbf u}(r))-\rho(\gamma'_{\mathbf u}(r),\gamma'_{\mathbf u}(r)).
	\end{equation}
	If $M$ is a D'Atri space, then $\theta$ is an even function, $M$ has cyclic parallel Ricci tensor, that is $\nabla_X\rho(X,X)\equiv 0$, and the scalar curvature $\tau$ of $M$ is constant. Having cyclic parallel Ricci tensor implies that the function $\rho(\gamma'_{\mathbf u},\gamma_{\mathbf u}')$ is constant on the domain of $\gamma_{\mathbf u}$. Hence the right hand side of \eqref{eq:tauS_from_theta}, and consequently both $\tau^S(r\mathbf u)$ and $\tau^S(r\mathbf u)\theta(r\mathbf u)$ are even functions of $r$.
	
	To prove that 
	$\ref{item:(iii)}$ implies $\ref{item:(i)}$, we first prove that  that $\ref{item:(iii)}$ implies the Ledger condition $L_3$. Choose an arbitrary unit tangent vector $\mathbf u\in SM$ and consider the functions $\theta(r\mathbf u)$, $\tau^S(r\mathbf u)$ and  $\tau^S(r\mathbf u)\theta(r\mathbf u)$. According to \cite[Thm.~4.4]{Chen_Vanhecke}, we have the power expansion
	\begin{align*}
	\tau^S(r\mathbf u)=\frac{(n-1)(n-2)}{r^2}+\left(\tau-\frac{2(n+1)}{3}\rho(\mathbf u,\mathbf u)\right)+\left(\nabla_{\mathbf u}\tau-\frac{n+2}{2}\nabla_{\mathbf u}\rho(\mathbf u,\mathbf u)\right)r+O(r^2),
	\end{align*}
which, combined with \eqref{eq:theta_sorfejtes}, yields
\begin{align*}
	\tau^S(r\mathbf u)\theta(r\mathbf u)=\frac{(n-1)(n-2)}{r^2}&{}+{}\left(\tau -\frac{n^2+n+6}{6}\rho(\mathbf u,\mathbf u)\right)\\ &+\left(\nabla_{\mathbf u}\tau-\frac{n^2+3n+14}{12}\nabla_{\mathbf u}\rho(\mathbf u,\mathbf u)\right)r+O(r^2).
\end{align*}
The coefficients of odd powers of $r$ have to vanish in the power expansion of an even function, so if  $\tau^S\theta$ is even, then the coefficient $\nabla_{\mathbf u}\tau-\frac{n^2+3n+14}{12}\nabla_{\mathbf u}\rho(\mathbf u,\mathbf u)$ of $r$ in its expansion must vanish for every $\mathbf u\in SM$. By Lemma \ref{lin}, this gives that $M$ satisfies the $L_3$ condition and has constant scalar curvature, hence $C(\mathbf u)=\rho(\gamma'_{\mathbf u}(r),\gamma'_{\mathbf u}(r))-\tau(\gamma_{\mathbf u}(r))$ is constant on the domain of $\gamma_{\mathbf u}$. Another important corollary of the third Ledger condition is that $M$ is a real analytic Riemannian manifold, therefore the functions $\theta(r\mathbf u)$ 
and $\tau^S(r\mathbf u)\theta(r\mathbf u)$ can be written as the sum of their Laurent series 
\[
\theta(r\mathbf u)=\sum_{k=0}^{\infty}a_k(\mathbf u)r^k,
\qquad \tau^S(r\mathbf u)\theta(r\mathbf u)=\sum_{k=-2}^{\infty}b_k(\mathbf u)r^k
\]
for small values of $r\neq  0$. Substituting these Laurent series into \eqref{eq:Steiner_for_spheres} and equating the coefficients of $r^k$, we obtain the following recursive equation for the coefficients $a_k$  assuming that we are given the coefficients $b_k$
\[
a_{k+2}=\frac{1}{(k+n+1)(k+n)}(Ca_k+b_k).
\]
This relation allows us to prove by an easy induction that if $\tau^S \theta$ is an even function, then $\theta$ is even as well, i.e., $a_{2k+1}=0$ for all natural number $k$. The base case $a_1=0$ is automatically fulfilled by \eqref{eq:theta_sorfejtes}. Assume $a_{2k-1}=0$. Then equation 
\[
a_{2k+1}=\frac{1}{(2k+n)(2k+n-1)}(Ca_{2k-1}+b_{2k-1})=0
\]
completes the induction step and $\ref{item:(i)}\iff \ref{item:(iii)}$ is proved.

Condition $\ref{item:(iii)}$ implies $\ref{item:(iv)}$ in an obvious way, since for any $\mathbf v\in S_p(r)$, the total scalar curvature of a hemisphere $\mathcal S^+(\mathbf v)$ is equal to the integral $\int_{S^+(\mathbf v)}\tau^S(\mathbf w)\theta(\mathbf w) \id \mathbf w$, which is exactly half the total scalar curvature of the sphere $\mathcal S_p(r)$ if $\tau^S\theta$ is an even function. The converse $\ref{item:(iv)}\Longrightarrow \ref{item:(iii)}$ follows from a classical result  of harmonic analysis on the sphere, as $\ref{item:(iv)}$ means that the hemispherical transformation of the restriction of the function  $\tau^S\theta$ onto any sphere $S_p(r)$ of small radius $r$ is constant and this implies by \cite[Prop.~3.4.11]{Groemer} that these restrictions are even functions.
\end{proof}

\begin{corollary}
	The scalar curvature function $\tau^S$ of any geodesic sphere of small radius in a D'Atri space is an even function. 
\end{corollary}
\begin{question}
	Assume that $\tau^S$ is an even function for a Riemannian manifold. Does it follow that the manifold is a D'Atri space? 
\end{question}
\section{3-dimensional D'Atri spaces and the total scalar curvature of tubes\label{sec:3dim}}
In this section, we strengthen Theorem \ref{thm:main1} in the $3$-dimensional case. The distinguished role of dimension three is due to the Gauss--Bonnet theorem, controlling the total scalar curvature of surfaces.

\begin{theorem}\label{thm:main2}
For a $3$-dimensional Riemannian manifold $(M,\langle\,,\rangle)$, the following conditions are equivalent:
\begin{enumerate}[label=(\roman*)]
	\item $M$ is a D'Atri space.
	\item The total scalar curvature of any geodesic hemisphere is equal to $4\pi$.
	\item The total scalar curvature of a tube of small radius about a regular curve is $0$.
	\item The total scalar curvature of a tube of small radius about a regular curve depends only on the length of the curve and the radius of the tube. 
\end{enumerate}
\end{theorem}
\begin{proof}

Theorem \ref{thm:main1} implies $(ii)\Longrightarrow (i)$. The total scalar curvature of a geodesic sphere of small radius in $M$ is $8\pi$ by the Gauss--Bonnet theorem. If $M$ is a D'Atri space, then by Theorem \ref{thm:main1}, the total scalar curvature of a geodesic hemisphere and its complementary hemisphere are equal, so they are both equal to $4\pi$. Thus, the equivalence $(i)\iff (ii)$ is proved. 

To prove $(ii)\Longrightarrow (iii)$, consider a tube  $\mathcal T_{\circ}(\gamma,r)$ of small radius $r$ about a regular parameterized curve $\gamma\colon [a,b]\to M$. We may assume without loss of generality that $\gamma$ is of unit speed. The union of the tube $\mathcal T_{\circ}(\gamma,r)$ and the hemispheres $\mathcal S^+(-r\gamma'(a))$ and $\mathcal S^+(r\gamma'(b))$ is the image of a piecewise smooth $\mathcal C^1$-immersion of a ``capsule'' homeomorphic to a sphere into $M$ so its total scalar curvature is $8\pi$ by the Gauss--Bonnet theorem. On the other hand, assumption $(ii)$ impies that the total scalar curvature of the union $\mathcal S^+(-r\gamma'(a))\cup \mathcal S^+(r\gamma'(b))$ is also $8\pi$, therefore the total scalar curvature of the tube $\mathcal T_{\circ}(\gamma,r)$ must be $0$. Conversely, assume that the total scalar curvature of any tube vanishes. Then computing the total scalar curvature of the immersed capsule constructed above we obtain that the sum of the total scalar curvatures of the geodesic hemisheres $\mathcal S^+(-r\gamma'(a))$ and $\mathcal S^+(r\gamma'(b))$ equals $8\pi$. Let $\mathbf u\in SM$ be an arbitrary unit vector and choose the regular curve $\gamma$ so that  $-\gamma'(a)=\gamma'(b)=\mathbf u$. Then $\mathcal S^+(-r\gamma'(a))=\mathcal S^+(r\gamma'(b))=\mathcal S^+(r\mathbf u)$, therefore $\mathcal S^+(r\mathbf u)$ must have total scalar curvature $4\pi$ for any small radius $r$.

Condition $(iv)$ is obviously weaker than $(iii)$. If condition $(iv)$ holds, then there exists a function $f\colon (0,r_0)\to\mathbb R$ such that the total scalar curvature of a tube of small radius $r$ about a regular curve $\gamma\colon [a,b]\to M$ of length $l_{\gamma}$ equals $f(r)l_{\gamma}$. Choosing an arbitrary smoothly closed regular curve $\gamma$, the tubes of small radii about $\gamma$ are immersed tori, so their total scalar curvature vanish by the Gauss--Bonnet theorem. This means that the function $f$ must vanish around $0$, hence  $(iv)\Longrightarrow (iii)$.
\end{proof}

\begin{theorem}\label{thm:main3}
	Assume that the $3$-dimensional Riemannian manifold $(M,\langle\,,\rangle)$ has the property that the total scalar curvature of a cylinder of small radius $r$ about a geodesic segment $\gamma$ depends only on the radius $r$ and the length of $\gamma$. 
\begin{enumerate}[label=(\roman*)]
\item	Then there is a number $a\in \mathbb R$ and a smooth function $b\colon SM\to \mathbb R$ such that for any geodesic curve $\gamma_{\mathbf u}$ with initial velocity $\gamma'_{\mathbf u}(0)=\mathbf u\in SM$, we have  
\begin{equation}\label{eq:K(normal)}
K(\nu(t))=at^2+b(\mathbf u)t+ K(\nu(0)),
\end{equation}
where $K(\nu(t))$ is the sectional curvature in the direction of the normal plane $\nu(t)\subset T_{\gamma_{\mathbf u}(t)}M$ of $\gamma_{\mathbf u}$ at $\gamma_{\mathbf u}(t)$. 
\item If we assume also that $M$ is complete and the sectional curvature of $M$ is bounded, (e.g., if $M$ is compact, or homogeneous), then $M$ is a D'Atri space.
\end{enumerate}
\end{theorem}

\begin{proof}
The initial terms of the power expansion of the total scalar curvature $T_{\gamma}(r)$ of a tube of small radius $r$ about a unit speed curve $\gamma\colon [a,b]\to \tilde M$ were computed explicitely by L.~Gheysens and L.~Vanhecke \cite[Thm.~5.1]{Gheysens_Vanhecke} in any $n$-dimensional Riemannian manifold $\tilde M$. Their formula has the form
	\begin{align*}
		T_{\gamma}(r)=c_{n-2}r^{n-4}\int_a^b\{(n-3)(n-2)+A(n)r^2+B(n)r^4+O(r^6)\}\id t,
	\end{align*}
	where $c_{n-2}$ is the volume of the unit sphere in the $(n-1)$-dimensional Euclidean space, 
	\[A(n)=-\frac{n-3}{6(n-1)}\{(n-4)\tau+(n+2)\rho_{11}\}(\gamma(t)),\]
	
	\begin{align*}
		B(n)={}&\frac{1}{n^2-1}\Big\{
		\frac{n^2-9n+2}{72}\tau^2+\frac{n^2+3n+17}{45}\|\rho\|^2-\frac{(n+1)(n+2)}{120}\|R\|^2\\
		&-\frac{(n-3)(n-4)}{20}\Delta\tau-\frac{(n+6)(n-3)}{40}\Delta\rho_{11}+\frac{11n^2-27n+142}{120}\nabla_{11}^2\tau\\
		&+\frac{(n-4)(n+1)}{36}\tau\rho_{11}-\frac{7n^2+21n-46}{180}\sum_{i,j\geq 2}\rho_{ij}R_{1i1j}-\frac{n^2+3n-58}{120}\rho_{11}^2\\
		&-\frac{7n^2+21n+194}{120}\nabla_{11}^2\rho_{11}-\frac{(n+1)(n+2)}{36}\sum_{i,j\geq 2}R_{1i1j}^2\\
		&+\frac{n^2+3n+62}{180}\sum_{i\geq 2}\rho_{1i}^2-\frac{(n+1)(n+2)}{60}\sum_{i,j,k\geq 2}R_{1ijk}^2+\frac{n^2-3n+8}{6}\nabla_{\gamma''}\tau\\
		&-\frac{n^2+3n+14}{6}\nabla_{1}\rho_{1\gamma''}-\frac{n^2+3n+14}{12}\nabla_{\gamma''}\rho_{11}
		\Big\}(\gamma(t)),
	\end{align*}
and the tensor coordinates are taken with respect to an orthonormal frame $E_1=\gamma',E_2,\dots,E_n$ along $\gamma$.

In particular, $c_1=2\pi$, $A(3)=0$, and using the identity $\|R\|^2=4\|\rho\|^2-\tau^2$, valid in any $3$-dimensional Riemannian manifold, a straightforward computation shows that
\begin{align*}
	B(3)={}&\frac{1}{6}\Big\{
	\nabla_{11}^2\tau-2\nabla_{11}^2\rho_{11}+\nabla_{\gamma''}\tau -4\nabla_{1}\rho_{1\gamma''}-2\nabla_{\gamma''}\rho_{11}
	\Big\}(\gamma(t)).
\end{align*}
In the special case when $\gamma$ is a geodesic curve, all the terms containing the acceleration $\gamma''$ disappear, thus, for the total scalar curvature $T_{\gamma}(r)$ of a cylinder of small radius about a geodesic segment $\gamma\colon [a,b]\to M$ lying in a $3$-dimensional manifold $M$, we have 
\[
	T_{\gamma}(r)=2\pi \int_a^b\Big\{\big\{\nabla_{11}^2\tau-2\nabla_{11}^2 \rho_{11}
	\big\}(\gamma(t))\frac{r^3}{6}+O(r^5)\Big\}\id t.
\]
This formula implies that if the total scalar curvature of a cylinder depends only on the radius and the length of the axis of the cylinder, then the coefficient $\hat a=\big\{\nabla_{11}^2\tau-2\nabla_{11}^2 \rho_{11}
\big\}(\gamma(t))$ of $r^3/6$ must be a constant independent of the geodesic $\gamma$ and the parameter $t$. Set $a=\hat a/4$.

Now let $\gamma_{\mathbf u}$, ($\mathbf u \in SM$) be an arbitrary unit speed geodesic in $M$, and let $E_1=\gamma_{\mathbf u}'$, $E_2$, $E_3$ be a parallel orthonormal frame along $\gamma_{\mathbf u}$, $\sigma_{ij}(t)\subset T_{\gamma_{\mathbf u}(t)}M$ the plane spanned by $E_i(t)$ and $E_j(t)$. Then the sectional curvature in the direction of the normal plane $\nu=\sigma_{23}$ can be expressed as 
\[K(\nu)= K(\sigma_{23})= \big(K(\sigma_{12})+ K(\sigma_{23})+ K(\sigma_{31})\big)-\big(K(\sigma_{12}) +K(\sigma_{31})\big)=\frac{1}{2}\tau\circ\gamma_{\mathbf u}-\rho(E_1,E_1).
\]
Differentiating this equation twice with respect to the curve parameter, and using the fact that the vector field $E_1=\gamma'_{\mathbf u}$ is parallel along $\gamma_{\mathbf u}$, we obtain
\[
K(\nu)''= \left(\frac{1}{2}\nabla^2_{11}\tau-\nabla^2_{11}\rho(E_1,E_1)\right)=\frac{\hat a}{2}=2a. 
\]
Thus $K(\nu(t))$ must be a polynomial function of $t$ of degree at most $2$ with leading term $at^2$. In particular, \eqref{eq:K(normal)} holds with a suitably chosen coefficient $b(\mathbf u)$. This proves $(i)$.

To prove $(ii)$, assume $M$ has bounded sectional curvature. Then for any choice of $\gamma_{\mathbf u}$, $K(\nu(t))$ is a bounded polynomial function defined on the whole real line, hence it is constant. Consequently, it has vanishing derivative
\[
\frac{d}{dt}K(\nu(t))= \left(\frac{1}{2}\nabla_{1}\tau-\nabla_{1}\rho(E_1,E_1)\right)=0.
\] 
Evaluating this equation at $t=0$, we obtain $\frac{1}{2}\nabla_{\mathbf u}\tau-\nabla_{\mathbf u}\rho(\mathbf u,\mathbf u)=0$ for any $\mathbf u\in SM$ and by Lemma \ref{lin}, we conclude that $M$ satisfies the Ledger condition $L_3$. H.~Pedersen and P.~Tod \cite{Pedersen_Tod} proved that $3$-dimensional Riemannian manifolds satisfying the third Ledger condition are D'Atri spaces,  so $M$ is a D'Atri space.
\end{proof}

\section{Acknowledgements}
The first and third authors were supported by the National Research Development and Innovation Office (NKFIH) grant No K-128862. During the research they also enjoyed the hospitality of the Alfr\'ed R\'enyi Institute of Mathematics as guest researchers. 

The first author has received funding also from the European Research Council (ERC) under the European Union's Horizon 2020 research and innovation programme (grant agreement No 741420).

The third author was also supported by the National Research Development and Innovation Office (NKFIH) grant No KKP-133864.

\bibliographystyle{acm}
\bibliography{3dim}
\end{document}